\documentclass[leqno]{amsart}

\usepackage{fancyhdr}
\usepackage{layout}
\usepackage{caption} 
\usepackage{amsthm,amsxtra}
\usepackage{amssymb}
\usepackage{amsmath}
\usepackage{amsfonts}
\usepackage[pdftex]{graphicx} 
\usepackage{anyfontsize}

\usepackage{latexsym}
\usepackage{mathrsfs}
\usepackage{enumerate}
\usepackage{verbatim}
\usepackage{lipsum}

\usepackage{here}
\usepackage{amscd,amstext}
\usepackage{graphicx}
\usepackage{epstopdf}

\DeclareMathOperator{\as}{as}
\DeclareMathOperator{\andd}{and}

\newcommand{\bR}{\mathbb R}

\newcommand{\bs}{\mathbf s}

\newcommand{\bw}{\mathbf w}

\newcommand{\bfP}\bfT
\newcommand{\bfT}{\mathbf T}

\newcommand{\Tr}{\mathrm{Tr}}

\newcommand{\cL}{\mathcal L}

\newcommand{\nti}{{n\to\infty}}

\newcommand{\vphi}{\varphi}
\newcommand{\veps}{\varepsilon}

\numberwithin{equation}{section}

\newtheorem{thm}{Theorem}[section]
\newtheorem{defn}[thm]{Definition}
\newtheorem{lem}[thm]{Lemma}
\newtheorem{coro}[thm]{Corollary}
\newtheorem{prop}[thm]{Proposition}

\newtheorem{nota}[thm]{Notation}

\theoremstyle{definition}

\newtheorem{exam}[thm]{Example}

\title{On the L\'evy constants of Sturmian continued fractions}

\author{Yann Bugeaud}
\address{Universit\'e de Strasbourg, IRMA, CNRS, UMR 7501, 7 rue Ren\'e Descartes, 67084, Strasbourg, France}
\email{yann.bugeaud@math.unistra.fr}

\author{Dong Han Kim}
\address{Department of Mathematics Education, Dongguk University--Seoul, 30 Pildong-ro 1-gil, Jung-gu, Seoul, 04620 Korea}
\email{kim2010@dongguk.edu}

\author{Seul Bee Lee}
\address{Centro di Ricerca Matematica Ennio de Giorgi, Scuola Normale Superiore, Piazza dei Cavalieri 3, 56126 Pisa, Italy}
\email{seulbee.lee@sns.it}



\begin{document}

\begin{abstract}
The L\'evy constant of an irrational real number is defined by the exponential growth rate 
of the sequence of denominators of the principal convergents in its continued fraction expansion.
Any quadratic irrational has an ultimately periodic continued fraction expansion and it is well-known that 
this implies the existence of a L\'evy constant. 
Let $a, b$ be distinct positive integers. 
If the sequence of partial quotients of an irrational real number 
is a Sturmian sequence over $\{a, b\}$, 
then it has a L\'evy constant, which depends only on $a$, $b$, and the slope of the Sturmian sequence, 
but not on its intercept. 
We show that the set of L\'evy constants of irrational real numbers whose 
sequence of partial quotients is periodic or Sturmian is equal to the whole interval $[\log ((1+\sqrt 5)/2 ), + \infty)$.
\end{abstract}

\maketitle

\section{Introduction and main results}

Let $\alpha$ be an irrational real number. 
Let $[a_0; a_1, a_2 \ldots ]$ be its continued fraction expansion and, for $n \ge 1$, let $P_n (\alpha) / Q_n (\alpha)$ 
denote the rational number $[a_0; a_1, a_2, \ldots , a_n]$, called the $n$-th partial quotient of $\alpha$. 
In 1936 Khintchine \cite{Khi36} proved that there exists a real number $\cL$ such that the sequence 
$((\log Q_n (\alpha)) / n )_{n \ge 1}$ converges to $\cL$ for almost all $\alpha$ and the same year 
L\'evy \cite{Levy36} proved that $\cL = \pi^2 / (12 \log 2)$. More generally, we say that $\alpha$ has  
a L\'evy constant if the sequence $((\log Q_n (\alpha)) / n )_{n \ge 1}$ converges, in which case we set 
$$
\cL(\alpha) = \lim_{n \to \infty} \frac 1n  \log Q_n (\alpha). 
$$
It is well known (see e.g. \cite{JaLi88}, \cite{Fai97}) that,
if $\alpha$ has a L\'evy constant, then we have 
\begin{equation}\label{Levylimit}
\cL(\alpha) = \lim_{n\to \infty}\frac{1}{n}\sum_{i=0}^{n-1} \log{([a_i;a_{i+1},\ldots])}.
\end{equation}
Not all irrational numbers $\alpha$ have a L\'evy constant. 
An obvious example is given by 
$[0; 10, 10^{2!}, \ldots , 10^{n!}, \ldots ]$, for which the above limit is equal to infinity. 
A more interesting example is the following. Let $a, b$ be distinct positive integers. 
Let 
$$
\xi_{a, b} := [0; a, b, a, a, b, b, b, b, a, a, \ldots]
$$
denote the real number whose first partial quotient is $a$ and, for $n \ge 0$,
whose $2^n + 1$-th up to $2^{n+1}$-th partial 
quotients are $b$ if $n$ is even and $a$ if $n$ is odd. An easy calculation (see at the end of Section 2 below) 
shows that $\xi_{a, b}$ does not have a L\'evy constant. 
Note that the set of real numbers which do not have a L\'evy constant 
has full Hausdorff dimension \cite[Theorem 3]{PoWe99}; see also \cite{BaSc00}. 

A quadratic real number has a L\'evy constants since
its continued fraction expansion is ultimately periodic;  
see \cite{JaLi88} or Lemma~\ref{lem:JaLi} below. 
In particular, for every positive integer $a$, we have
\begin{equation}\label{formula0}
\cL ([0; a, a, a, \ldots ]) = \log \frac{a + \sqrt{a^2 + 4}}{2}. 
\end{equation}
Since the quadratic real numbers are exactly the real numbers whose 
sequence of partial quotients is ultimately periodic, that is, the real numbers whose continued fraction expansion is the most simple, 
from the point of view of combinatorics on words, we may ask whether all real numbers whose sequence of partial 
quotients is sufficiently simple, in some sense, have a L\'evy constant. 
Let us be more precise. 
For an infinite word ${\bf w} = w_1 w_2 \ldots $ over the alphabet of positive integers, let
$$
p(n, {\bf w}) := \hbox{Card} \{ w_{i+1} \ldots w_{i+n} : i \ge 1 \}
$$
denote the number of distinct factors (subwords)  of length $n$ contained in ${\bf w}$. The function 
$n \mapsto p(n, {\bf w})$ is called the complexity function of $\bw$. 
Clearly, if $(w_j)_{j \ge 1}$ is unbounded, then $p(n, {\bf w})$ is infinite for every $n \ge 1$. 
In the sequel, we tacitly assume that $(w_j)_{j \ge 1}$ is bounded. 
If $(w_j)_{j \ge 1}$ is ultimately periodic, then the sequence $(p(n, {\bf w}))_{n \ge 1}$ is bounded
and $[0; w_1, w_2, \ldots ]$ is a quadratic real number, thus it has a L\'evy constant.
If $(w_j)_{j \ge 1}$ is not ultimately periodic, then the sequence $(p(n, {\bf w}))_{n \ge 1}$ 
satisfies $p(n, {\bf w}) \ge n+1$ for $n \ge 1$. Infinite words ${\bf w}$ for which $p(n, {\bf w}) = n+1$ 
for $n \ge 1$ do exist and are called Sturmian words. They can be described as follows.

Let $a, b$ be distinct positive integers. 
For a real number $\theta$ in $[0,1]$ and a real number $\rho$ with $0 \le \rho < 1$, set
\begin{equation}\label{eq:mechanical word}
s_{n}= 
\begin{cases} 
a &\text{if} ~ \lfloor n \theta + \rho\rfloor  - \lfloor (n-1) \theta + \rho \rfloor = 0, \\
b &\text{if} ~ \lfloor n \theta + \rho \rfloor  - \lfloor (n-1) \theta + \rho \rfloor = 1 
\end{cases}
\end{equation}
and 
\begin{displaymath}
s'_{n}=
\begin{cases} 
a &\text{if} ~  \lceil n \theta + \rho \rceil  - \lceil (n-1) \theta + \rho\rceil = 0, \\
b &\text{if} ~ \lceil n \theta + \rho \rceil  - \lceil (n-1) \theta + \rho\rceil = 1. 
\end{cases}
\end{displaymath}
We also set $\bs_{\theta,\rho} = (s_n)_{n\ge 1}$ and $\bs'_{\theta,\rho} = (s_n')_{n\ge 1}$
and call these infinite words 
\emph{mechanical words}. The real numbers $\theta$ and $\rho$ are called \emph{the slope} and \emph{the intercept}, respectively.
It is well-known (see e.g. \cite[Theorem 2.1.13]{Lo02}) 
that an infinite word is a Sturmian word if and only if it is a mechanical word with an irrational slope.   
Obviously, for a rational $\theta=p/q$, the sequences $\bs_{p/q,\rho}$ and $\bs'_{p/q,\rho}$ are periodic words with period $q$; see Definition~\ref{def:Chr}.

Berth\'e \cite{Be96} showed that the set of Sturmian words over $\{a, b\}$ 
\begin{equation}\label{eq:S_theta}
S_\theta = \{ \bs_{\theta,\rho} \, : \, \rho \ \in [0,1) \} \cup \{ \bs'_{\theta,\rho} \, : \, \rho \ \in [0,1) \}
\end{equation} with the same slope $\theta$ endowed with the 
left-shift map $\sigma : S_\theta \to S_{\theta}$ given by $\sigma(s_1s_2\dots) = s_2s_3 \dots$ is uniquely ergodic. 
Therefore, for $s_1s_2\dots$ in $S_\theta$, 
the Birkhoff sum $$\lim_{n\to\infty} \frac{1}{n}\sum_{i=1}^{n} \log{([s_i;s_{i+1},\ldots])}$$
converges to a value depending only on the slope $\theta$. 
By \eqref{Levylimit}, the real number $[0 ; s_1, s_2, s_3, \ldots]$ has a L\'evy constant, 
whose value depends only  on $\theta$. 
Combined with Cassaigne's description of quasi-Sturmian words \cite{Cassa98}, this yields    
the following result.  





\begin{thm}\label{thm:th1}
Let ${\bf w} = w_1 w_2 \ldots $ be an infinite word over the positive integers. If there exists an integer $k$
such that
$$
p(n, {\bf w}) \le n + k, \quad \hbox{for $n \ge 1$}, 
$$
then the real number $[0; w_1, w_2,\ldots]$ has a L\'evy constant. 
\end{thm}


We give a new, purely combinatorial proof of Theorem  \ref{thm:th1} in Section \ref{sec:exist}.

The real numbers $\xi_{a, b}$ defined above show that we cannot hope for a much better result than 
Theorem \ref{thm:th1}. 
Indeed, it is easy to see that the complexity function of the infinite word $\bw_{a,b}$ 
formed by the concatenation of its 
partial quotients satisfies
$2 n \le p(n, \bw_{a,b}) \le 3 n$, for $n \ge 1$ (a careful study yields more precise bounds).

In 1997 Faivre \cite{Fai97} showed that, for all real numbers $\ell$ with $\ell \ge \log((1 + \sqrt{5})/2)$,    
there exists an irrational real number $\alpha$ such that $\ell = \cL (\alpha)$. 
Wu \cite{Wu06}  proved that the L\'evy constants of quadratic irrationalities are dense in the interval 
$[ \log((1 + \sqrt{5})/2), + \infty)$. An alternative proof of Wu's result   
was given in 2008 by Baxa \cite{Baxa08}, who 
established a slightly stronger result, namely that, for all integers $a, b$ with $1 \le a < b$, 
the closure of the set 
$$
\{ \cL (\alpha) : \hbox{$\alpha$ quadratic with partial quotients in $\{a , b\}$} \}
$$
is equal to the whole interval $[\cL([0 ; \overline{a} ]) , \cL ([0 ; \overline{b} ]) ]$. 
Here and throughout the present paper, for positive integers $a_1, \ldots, a_r, a_{r+1}, \ldots , a_{r+s}$ 
and for an integer $a_0$, we denote by $[a_0 ; a_1, \ldots , a_r, \overline{a_{r+1}, \ldots , a_{r+s}}]$ the 
quadratic number whose sequence of partial quotients starts with $a_0, a_1, \ldots , a_r$, followed by its 
periodic part $a_{r+1}, \ldots , a_{r+s}$. 

Our main result is the following refinement of Baxa's result.  
A Sturmian (resp., mechanical) continued fraction is a continued fraction whose sequence of partial quotients is a Sturmian (resp., mechanical) sequence. 
Recall that any mechanical continued fraction is either Sturmian, or represents a quadratic number. 
It is known that Sturmian continued fractions are transcendental numbers \cite{ADQZ01}.

\begin{thm}  \label{thm2}
Let $a, b$ be integers with $1 \le a < b$.
The set of L\'evy constants of mechanical continued fractions with intercept 0 and written over the alphabet $\{a , b\}$ 
is equal to the whole interval $[\cL([0 ; \overline{a} ]) , \cL ([0 ; \overline{b} ]) ]$.   
\end{thm}

We display an immediate corollary.

\begin{coro}
For all real numbers $\ell$ with $\ell \ge \log((1 + \sqrt{5})/2)$, there exists a    
Sturmian continued fraction or a quadratic real number whose L\'evy constant is equal to $\ell$. 
Let $a$ be an integer with $a \ge 3$. Then, there exists a mechanical continued fraction $\alpha$  
over $\{1, a\}$ such that $\cL (\alpha) = \cL = \pi^2 / (12 \log 2)$. 
\end{coro}

The last statement of the corollary follows from the inequality $\cL < \log \frac{3 + \sqrt{13}}{2}$. 

To establish Theorem \ref{thm2} we define a function $f$ on $[0, 1]$ which associates to a given $\theta$   
in $[0, 1]$ the L\'evy constant of some mechanical continued fraction    
and we show that $f$ is increasing and continuous.    
It would be interesting to study more deeply the regularity of $f$.   

Our combinatorial proof of Theorem \ref{thm:th1} is given in Section \ref{sec:exist}, while Section \ref{sec:sec4}    
is devoted to the proof of Theorem \ref{thm2}.   
In passing, we establish in Proposition \ref{prop:phi} a new formula for the L\'evy constant of a 
quadratic number, which appears to be crucial in our proof.

\section{Any Sturmian continued fraction has a L\'evy constant}\label{sec:exist}

Let $K(a_1, \ldots, a_n)$ be the denominator of the rational number $[0;a_1, a_2, \ldots, a_n]$.
Then we have 
$$
\begin{pmatrix} a_1 & 1 \\ 1 & 0 \end{pmatrix} \ldots \begin{pmatrix} a_n & 1 \\ 1 & 0 \end{pmatrix} 
= \begin{pmatrix} K(a_1, \dots, a_n)  & K(a_1, \dots, a_{n-1})  \\ K(a_2, \dots, a_{n})  & K(a_2, \dots, a_{n-1})  \end{pmatrix}.
$$
Therefore, we have 
\begin{equation*}\begin{split}
& K(a_1, \ldots, a_{n+m}) \\
& = K(a_1, \ldots, a_n) K(a_{n+1}, \ldots, a_{n+m}) +  K(a_1, \ldots, a_{n-1})  K(a_{n+2}, \ldots, a_{n+m}) .
\end{split}\end{equation*}
Thus,
\begin{equation}\label{eq:split K}\begin{split}
 K(a_1, \ldots, a_n) K(a_{n+1}, \ldots, a_{n+m}) & < K(a_1, \ldots, a_{n+m}) \\
									& \le 2 K(a_1, \ldots, a_n) K(a_{n+1}, \ldots, a_{n+m}) .
\end{split}\end{equation}
Let us write $K(M) = K( b_1, b_2, \ldots, b_n)$ for a word $M = b_1 b_2 \ldots b_n$.
This is called the continuant of $M$.    

We give a proof that the real number $\xi_{a,b}$ defined in Section 1 has no L\'evy constant. 
By construction, for $m \ge 1$, its first $2^m$ partial quotients are followed by $2^m$
partial quotients equal to $a$ if $m$ is odd and equal to $b$ if $m$ is even. 
Let $(Q_k)_{k \ge 1}$ denote the sequence of denominators of its convergents. 
For $m \ge 2$, we have 
$$
Q_{2^m} K(a^{2^m}) \le Q_{2^{m+1}} \le 2 Q_{2^m} K(a^{2^m}), \quad  
\hbox{if $m$ is odd},
$$
and
$$
Q_{2^m} K(b^{2^m}) \le Q_{2^{m+1}} \le 2 Q_{2^m} K(b^{2^m}), \quad   
\hbox{if $m$ is even}. 
$$
Consequently, if $\xi_{a, b}$ has a L\'evy constant $\ell$, then $\ell$ must satisfy
$$
\ell = \frac{\ell}{2} + \frac{\cL([0, \overline{b}])}{2} =  \frac{\ell}{2} + \frac{\cL([0, \overline{a}])}{2},
$$
a contradiction to \eqref{formula0} with the assumption that $a$ and $b$ are distinct.

Let $\mathbf{w}$ be an infinite word. 
For $n \ge 1$, we denote by $F_{n}(\mathbf{w})$ the set of distinct factors (subwords) of $\mathbf{w}$ of length $n$ and 
observe that $p(n, {\bf w})$ is the cardinality of $F_{n}(\mathbf{w})$. We further let 
$F (\mathbf{w})$ denote the set of all finite factors of $\mathbf{w}$.

We point out that two Sturmian words have the same set 
of factors if and only if they have the same slope; see Proposition 2.1.18 in \cite{Lo02}.

We let $\alpha = [0;\bs] = [0; s_1, s_2, \ldots] $ denote   
the real number whose sequence of partial quotients 
is given by the Sturmian word $\bs = s_1 s_2 \ldots$.
The slope of the Sturmian word $\bs$ is the irrational real number $\theta = [0;1+d_1,d_2,\ldots]$. 
We may assume that $d_1 \ge 1$ since we can set $\theta' = 1-\theta$ otherwise.
%
%
We denote the principal convergents of $\alpha = [0;\bs]$ by ${P_i}/{Q_i}$ and
the principal convergents of $\theta = [0;1+d_1,d_2,\ldots]$ by ${p_i}/{q_i}$.  
Thoroughout this paper, the length of a finite word $M$, that is, the number of letters composing $M$, is    
denoted by $|M|$.    
For a finite word $M = b_1b_2\ldots b_n$, we
denote by $M^-$ its prefix $b_1\ldots b_{n-1}$ of length $n-1$. 

We define the sequence of words $(M_n)_{n \ge -1}$ by setting
\begin{equation}\label{eq:M_n}
M_n = M_{n-1}^{d_n}M_{n-2},
\end{equation}
where $M_{-1} = b$ and $M_{0} = a$.
Then it is known that 
\begin{equation}\label{eq:CF alg of c_theta}
\mathbf{s}_{\theta, \theta} = \lim\limits_{n\to\infty} M_n.  
\end{equation}
and that $|M_n| = q_n$, for $n \ge 0$; see \cite[Proposition 2.2.24]{Lo02}.

We observe that 
$$
(M_1M_0)^{--} = a^{d_1} = (M_0M_1)^{--}.
$$
and claim that, for $k\ge 2$, we have 
\begin{equation}\label{rmk:fact1}
M_kM_{k-1}^{--} = M_{k-1}M_k^{--}.
\end{equation}
This claim can be proved by induction, using that 
$$
M_kM_{k-1}^{--}  = M_{k-1}^{d_k}M_{k-2}M_{k-1}^{--} = M_{k-1}^{d_k}M_{k-1}M_{k-2}^{--}= M_{k-1}M_k^{--}.
$$
By \eqref{eq:CF alg of c_theta}, the Sturmian word $\mathbf{s}_{\theta, \theta}$   
can be written as the concatenation of words $M_k$ and $M_{k-1}$, without two consecutive copies of $M_{k-1}$.  
Let $M$ be a factor in $F_{q_k-1}(\mathbf{s})$.    
Since $F_n(\mathbf{s}_{\theta, \theta}) = F_n(\bs)$ for all $n$,  
combined with \eqref{rmk:fact1}, this shows that
\begin{equation}\label{fact1}
\text{any factor $M$ of length $q_k - 1$ of $\mathbf s$ is a factor of $M_k (M_k)^{- -}$.}
\end{equation}
Moreover, any factor of $\bs$ of arbitrary length can be decomposed as follows.   
\begin{lem}\label{factor}
Let $\bs$ be a Sturmian word.
Let $M$ be a factor of $\mathbf s$ of length $n$ with $q_k \le n \le q_{k+1} -1$ and $(M_k)_{k\ge -1}$ be the sequence of finite words defined as in \eqref{eq:M_n}.
Then, we have either
\begin{enumerate}
\item[(a)] $M$ is a factor of $M_k M_k \dots M_k M_{k-1}$, or
\item[(b)] $M = UV$, where $U$ is a suffix of $M_{k+1}$ and $V$ is a prefix of $M_{k+1}$  with $|V| \ge q_k -1$.
\end{enumerate}
\end{lem}

\begin{proof}
By \eqref{fact1}, $M$ is a factor of $M_{k+1}$ or $M = UV$ where $U$ is a suffix of $M_{k+1}$ and $V$ is a prefix of $M_{k+1}$. If $|V| \le q_{k}-2$, then $V$ is a prefix of $M_k^{--}$.
Thus, $M = UV$ is a factor of $M_{k+1} {M_{k}}^{--} = M_k {M_{k+1}}^{--} $ and also a factor of $M_k M_k \dots M_k M_{k-1}$.
\end{proof}


The next Proposition asserts that the ratio of the continuants of two factors of the same length $n$    
of a Sturmian word is bounded from above by some power of $n$.  
This is the key auxiliary result to apply Fekete’s lemma and derive the existence of a L\'evy constant.

\begin{prop}\label{prop1}
Let $\bs$ be a Sturmian word over $\{a,b\}$ and $ c =\max\{b/a,a/b\}$.
For any factor $M$, $M'$ of $\bs$ of length $n$ with with $q_k \le n \le q_{k+1} -1$, we have 
$$
K(M) \le  2^{2k} c K(M').
$$
\end{prop}

\begin{proof}
Suppose first that $k = 0$. Recall that $q_0 = 1$ and $q_1 = d_1+1$. 
Let $n$ be an integer with $q_0\le n \le q_1-1$. 
A factor of $\bs$ of length $n$ is a factor of $M_1M_1^{--} = a^{d_1}ba^{d_1-1}$.
Then any factor of $\bs$ of length $n$ is either $a^n$, or $a^{n_1} b a^{n_2}$ with $n_1 + n_2 = n-1$.

Assume that $a<b$.
Clearly, $K(a^n)<K(a^{n_1}ba^{n_2})$.
Since
\begin{equation*}\begin{split}
 aK(a^{n_1}b) & = a(bK(a^{n_1})+K(a^{n_1-1})) \\
		& \le b(aK(a^{n_1})+K(a^{n_1-1})) =  bK(a^{n_1+1}),
\end{split}\end{equation*}
it is easy to check that,
$$  aK(a^{n_1} b a^{n_2}) \le b K(a^n) .$$
Thus, $aK(a^{n_1}ba^{n_2})\le bK(a^n) \le bK(a^{n_1'}ba^{n_2'})$ with $n_1'+n_2' = n-1$.
Similarly, for the case of $a>b$, we check that 
$bK(a^{n_1}ba^{n_2})\le bK(a^n) \le aK(a^{n_1'}ba^{n_2'})$.    
Hence, the proposition holds for every positive integer $n$ at most equal to $q_1 - 1$.

We argue by induction. Let $k$ be a positive integer and
suppose that the proposition holds for every positive integer $n$ at most equal to $q_k - 1$. 
Let $n$ be an integer with $q_k \le n \le q_{k+1} -1$.
Let $M, M'$ be two factors of $\bs$ of length $n$. 
By Lemma~\ref{factor}, we distinguish the following cases:

\noindent Case (i): Both of $M, M'$ are factors of $M_k M_k \ldots M_k M_{k-1}$.

Since $M_k M_k \ldots M_k M_{k-1}$ is a periodic word with period $q_k$, there exist factors $N, N'$
of $\bs$ such that $NM = M'N'$ and $|N| = |N'| < q_k$.
Therefore, 
\begin{align*}
K(M) = \frac{K(N)K(M)}{K(N)} & \le \frac{K(NM)}{K(N)} \\
&\le c 2^{2(k-1)}\frac{K(M'N')}{K(N')} \\ 
& \le c 2^{2k-1} \frac{K(M')K(N')}{K(N')} \le c 2^{2k-1} K(M').
\end{align*}

\noindent Case (ii): 
Let $M = UV$ and $M' = U'V'$ where $U, U'$ are (possibly empty) suffixes of $M_{k+1}$ and $V, V'$ are (possibly empty) prefixes of $M_{k+1}$.
We may assume that $|U| < |U'|$. Define the words $N, N'$ by $U' = N U$ and $V = V' N'$. 
It follows from Case (i) that
$$K(N') \le c \cdot 2 \cdot 4^{k-1} K(N).$$
Therefore, 
\begin{align*}
K(M) = \frac{ K(N) K(UV)}{K(N)}  \le \frac{K(N UV)}{K(N)} & \le c 2^{2k-1}  \frac{K(U'V'N')}{K(N')} \\
&\le c 2^{2k-1}  \frac{2 K(U'V')K(N')}{K(N')} = c 2^{2k} K(M').
\end{align*}

\noindent Case (iii): Assume that $M$ is a factor of $M_k \ldots M_k M_{k-1}$ and 
$M' = U V$, where $U$ is a suffix of $M_{k+1}$ and $V$ is a prefix of $M_{k+1}$ with $|V| \ge q_k-1$.
Write $M = N M_k \ldots M_k N'$ where $N$ is a nonempty prefix of $M_k$ and  $N'$ is a possibly empty suffix of $M_k$.

\begin{enumerate}

\item If $|N'| \ge q_{k-1}$, then  $N' = M_{k-1} N''$ is a prefix of $(M_k M_{k-1} )^{--} = (M_{k-1} M_k)^{--}$.
Hence, 
$$M = N M_k \ldots M_k N' = N M_k \ldots M_k M_{k-1} N'',$$
where $N M_k \ldots M_k M_{k-1}$ is a suffix of $M_{k+1}$ and $N''$ is a prefix of $M_k$, which is a prefix of $M_{k+1}$.
We apply the argument of Case (ii).

\item If $|N'| \le q_{k-1}-2$, then $M_k N'$ is a prefix of $(M_k M_{k-1} )^{--} = (M_{k-1} M_k)^{--}$.
Define $N''$ by  $$M_k N' = M_{k-1} N''.$$ 
Since $n \ge q_k$, we get that $N''$ is a suffix of $M$.
We write $M = W N''$, where  $W$ is a suffix of $M_k \ldots M_k M_{k-1}$.
Since $N''$ is a prefix of $M_k$, which is a prefix of $M_{k+1}$, we apply the argument of  Case (ii).

\item Suppose that $|N| \ge 2$. Put $M_k = V'N$. Then $V'$ is a prefix of $V$.
Let $V'', W$ be the words defined by $V = V' V''$ and $M = V'' W$.
Since $U V'$ is a factor of  $M_{k+1} M_{k}^{--} = M_k M_{k+1}^{--}$ and $M_k M_{k+1}^{--}$ 
is periodic with period $q_k$, we have
$$K(W) \le c 2^{2k-1} K(UV'),$$
in a similar way as in Case (i).
Thus, 
\begin{equation*}\begin{split}
K(M) \le 2K(V'')K(W) & \le 2K(V'') \cdot c 2^{2k-1} K(UV') \\
			&  \le c 2^{2k} K(U V' V'' ) = c 2^{2k} K(M'). 
\end{split}\end{equation*}

\item The remaining case is the case where $|N| = 1$ and $|N'| = q_{k-1} -1$.
Then, for some $d\ge 0$, 
\begin{equation*}
M =  
\begin{cases}
a (M_k)^{d+1} (M_{k-1})^-,  &\text{ if $k$ is even}, \\
b (M_k)^{d+1} (M_{k-1})^-,  &\text{ if $k$ is odd},
\end{cases}
\end{equation*}
since $M_k$ is ending with $a$ if $k$ is even and with $b$ otherwise.
Note that $|M'| = |M| = (d+1) q_k + q_{k-1}$.
Since $U$ is a suffix of $(M_{k})^{d_{k+1}}M_{k-1}$ 
and $V$ is a prefix of $(M_{k})^{d_{k+1}}(M_{k-1})^{--}$, we have
$$VU = (M_k)^{d+1} M_{k-1}.$$
If $k$ is even, then
\begin{align*}
K(M) & \le 4 K(a) K( V) K(U^-) \\
& \le 4 c K(b) K( V ) K(U^-)  \le 4 c K(U^- b V ) = 4 c K(M').
\end{align*}
The case of odd $k$ is symmetric. \qedhere
\end{enumerate}
\end{proof}

For our alternative proof of the existence of a L\'{e}vy constant for any Sturmian continued fraction, we   
apply Fekete's lemma.  

\begin{lem}[Fekete]
If a sequence $(a_n)_{n \ge 1}$ of positive real numbers is subadditive, that is, if
it satisfies $a_{n+m}\le a_n+a_m$ for every integers $n, m$, then the sequence $(a_n/n)_{n \ge 1}$ converges 
and
$$
\lim_{n \to \infty} \frac{a_n}{n} = \inf_{n\ge 1} \frac{a_n}{n}.   
$$  
\end{lem}

We have now all the material to establish the following theorem. 

\begin{thm}\label{thm:Sturm}
Let $\alpha  = [0;\bs]$ be a Sturmian continued fraction. 
Then, $\alpha$ has a L\'evy constant $\cL(\alpha)$ and 
$$\cL(\alpha) = \lim\limits_{\nti}\frac{1}{n}{\log Q_n(\alpha)}.$$
\end{thm}

\begin{proof}
We apply Proposition \ref{prop1}. 
Let $\theta = [0;1+d_1,d_2,\ldots]$, where $d_1 \ge 1$, denote the slope of $\bs$.
For $k \ge 1$, let $q_k$ denote the denominator of $\theta$.  
Let $k$ be a non-negative integer.
Let $n$ be an integer with $q_k \le n \le q_{k+1} -1$.
Let $M,  M'$ be factors of $\bs$ of length $n$. 
Since $q_k \ge 2^{k/2}$, 
we have 
\begin{equation} \label{MM}
K(M) \le 4^{k} c K(M') \le c (q_k)^4 K(M') \le  c n^4 K(M'). 
\end{equation}
Set $A_n = 2^4 c n^4 K(s_1,s_2, \dots , s_n )$.
Then for $m \le n$ we have
\begin{align*}
A_{n+m} &=2^4 c (n+m)^4 K(s_1,s_2, \dots , s_{n+m} )  \\
&\le 2^4 c (2n)^4 2  K(s_1,s_2, \dots , s_{n} ) K(s_{n+1},s_{n+2}, \dots , s_{n+m} ) \\
& \le 2^4 c n^4  K(s_1,s_2, \dots , s_{n} ) 2^4 c m^4 K(s_{1},s_{2}, \dots , s_{m} ) = A_n A_m.
\end{align*}

By Fekete's lemma, 
the following limits exist and are equal:
$$\lim\limits_\nti \frac{1}{n}\log{A_n} = \lim\limits_\nti \frac{1}{n}\log{(2^4 cn^4Q_n(\alpha))} 
= \lim\limits_\nti \frac{1}{n}\log{Q_n(\alpha)}.$$
This proves that $\alpha$ has a L\'evy constant, which, by \eqref{MM} and the fact that two Sturmian words  
with the same slope have the same set    
of factors, does not depend on the intercept of $\bs$.
\end{proof}

\begin{proof}[Completion of the proof of Theorem \ref{thm:th1}]
Let ${\mathbf w} = w_1 w_2 \ldots$ be an infinite word defined over the positive integers such that    
the sequence $(p(n,{\mathbf w}) - n)_{n \ge 1}$ is bounded and ${\mathbf w}$ is not ultimately periodic. 
Since the function $n \mapsto p(n,{\mathbf w})$ is increasing,  
the sequence $(p(n,{\mathbf w}) - n)_{n \ge 1}$ 
of positive integers is eventually constant. Thus, there exist positive integers $k$ and $n_0$ such that
\begin{equation} \label{Qsturm}
p(n,{\mathbf w}) = n + k, \quad \text{ for } n \ge n_0.       
\end{equation}
Infinite words satisfying \eqref{Qsturm} are called \emph{quasi-Sturmian words}. 
We use a relation between Sturmian words and quasi-Sturmian words. This is a result of Paul \cite{Pa75} and Coven \cite{Co74} that has been restated several times; see \cite[Remarque 3]{Al00} for an historical account. We quote here Cassaigne's statement from \cite[Proposition 8]{Cassa98}:
There are a finite word $W$, a Sturmian word $\bs$ defined over $\{a, b\}$ and a 
morphism $\phi$ from $\{a, b\}^*$ into the set of positive integers such that
$\phi (ab) \not= \phi (ba)$ and 
$$  
{\mathbf w} = W \phi (\bs).
$$

We briefly explain that Proposition \ref{prop1} can be suitably extend to the word ${\mathbf w}$.

Put $c_\phi = \max \{ K(\phi(a))/ K(\phi(b)), K(\phi(b))/ K(\phi(a))\}$.
For any nonnegative integers $n_1,n_2,n'_1,n'_2$ with $n_1 + n_2 = n'_1 + n'_2$, 
by \eqref{eq:split K}, we have
\begin{align*}
K(\phi(a^{n_1}ba^{n_2})) &\le 4 K (\phi(a^{n_1})) K (\phi(b)) K (\phi(a^{n_2})) \\
&\le 4  c_\phi K (\phi(a^{n_1 + n_2+1})) \le 4^2 c_\phi K(\phi(a^{n'_1}ba^{n'_2})),
\end{align*}
or
$$
K(\phi(a^{n_1}ba^{n_2})) \le 4 K (\phi(a^{n_1 + n_2+1})) \le 4^2 c_\phi K(\phi(a^{n'_1}ba^{n'_2})), 
$$
depending on the fact that $K(\phi(a)) \le K(\phi(b))$ or $K(\phi(b)) \le K(\phi(a))$.  
Therefore, by replacing $K(M),~ K(M')$ with $K(\phi(M))$, $K(\phi(M'))$ in the proof of Proposition~\ref{prop1}, 
we conclude that, for any factors $N, ~ N'$ of $\bs$ with $q_k \le |N|=|N'| \le q_{k+1}-1$, we get 
$$
K(\phi(N)) \le 4^{k+2} c_\phi K(\phi(N')).
$$
Set  
$h = \max \{ | \phi(a)|, | \phi(b) |\}$.
Let $M, ~ M'$ be factors of the same length of $\phi(\bs)$.
Let $L$ (resp., $N$) be the word of minimal (resp., maximal) length   
such that $M$ is a factor of $\phi(L)$ (resp., $\phi(N)$ is a factor of $M'$).
Since $\bs$ is a balanced word, we have $|L|-|N| \le 6$.
Setting $\tilde c = \max\{ K(M) \, | \, M = \phi(N) \text{ for } |N|= 6  \}$, we get
$$
K(M) \le K(\phi(L)) \le 2 \tilde c \cdot 4^{k+2} c_\phi K(\phi(N)) \le 2 \tilde c \cdot 4^{k+2} c_\phi K(M'),
$$
and we conclude as in the proof of Theorem \ref{thm:Sturm}. 
We observe that the L\'evy constant of $[0; w_1, w_2, \ldots]$ depends only on the slope   
of the Sturmian word $\bs$.  
\end{proof}

\section{The L\'evy constants of quadratic numbers}\label{sec:sec3}

The L\'evy constants of quadratic numbers have been discussed in many papers. 
Jager and Liardet \cite{JaLi88} applied a result of Kiss \cite{Ki82} to compute them; 
see also the papers of Lenstra and Shallit \cite{LeSh93} and of Belova and Hazard \cite{BeHa17}.

\begin{lem}[Jager and Liardet]\label{lem:JaLi} 
If $\alpha$ is the quadratic irrational whose continued fraction expansion is 
given by $[a_0;a_1,\ldots,a_r,\overline{a_{r+1},\ldots,a_{r+s}}]$, then
$$
\cL(\alpha) = \frac{1}{s} \log \frac{t + \sqrt{t^2 - (-1)^{s} 4}}{2}, 
$$
where $t$ is the trace of 
$$
\begin{pmatrix} a_{r+1} & 1 \\ 1 & 0 \end{pmatrix} \begin{pmatrix} a_{r+2} & 1 \\ 1 & 0 \end{pmatrix} \ldots \begin{pmatrix}a_{r+s} & 1 \\ 1 & 0 \end{pmatrix}. 
$$
\end{lem}

The key argument in the proof is the fact that, denoting by $(P_n / Q_n)_{n \ge 1}$ the 
sequence of convergents to 
$$
\alpha = [0 ; \overline{a_{r+1}, \ldots , a_{r+s}}],
$$
we have $Q_{n+2s} = t Q_{n+s} - (-1)^{s} Q_n$, for $n \ge 1$, where $t$ is as in Lemma \ref{lem:JaLi}.

Our first auxiliary result in this section is an alternative expression for the 
L\'evy constant of a quadratic irrational number. 
We first need to introduce some notation.

\begin{nota}
\begin{enumerate}
\item
For real numbers $a_1, \ldots , a_n$, set
$$
T (a_1, \dots, a_n) 
= \Tr \left(  \begin{pmatrix} a_1 & 1 \\ 1 & 0 \end{pmatrix} \ldots \begin{pmatrix} a_n & 1 \\ 1 & 0 \end{pmatrix} \right).
$$

\item 
For $n\ge 1$, we define polynomials $\bfT_n(x)$ by 
$$\bfT_n(x) = \Tr \left( X^n \right) = T(x, \ldots , x),$$
where
$$ X =  \begin{pmatrix}x & 1 \\ 1 & 0 \end{pmatrix}.$$
\end{enumerate}
\end{nota}

\begin{lem}\label{lem:first}
Let $U, V$ be $2 \times 2$ matrices. 
If $U = WV$ or $VW$, then we have
\begin{equation*}
\Tr( VU ) = \Tr( UV ) = \Tr( U ) \Tr( V ) - \det(V) \Tr( W ).
\end{equation*}
In particular, for any positive integers $q, q'$ with $q \ge q'$, we have 
\begin{equation}\label{eq1}
\Tr( X^{q+q'}) = \Tr( X^q ) \Tr( X^{q'}) - (-1)^{q'} \Tr(X^{q-q'}). 
\end{equation}
\end{lem}

\begin{proof}
If $U = WV$, then by the Cayley-Hamilton Theorem, we have
\begin{equation}
UV = WV^2 = W ( \Tr(V) V - \det(V) I ) = \Tr(V)U - \det(V) W,
\end{equation}
where $I$ is the $2 \times 2$ identity matrix. Then we obtain the result by taking the trace. 
The case $U = VW$ is similar.
Finally, taking $U = X^q$ and $V=X^{q'}$, we immediately derive \eqref{eq1}. 
\end{proof}

Using \eqref{eq1} we have 
\begin{equation}\label{eq1p}
\bfT_{n+1} (x) =  x \bfT_n (x) + \bfT_{n-1} (x), \quad  \bfT_0 (x) =  2, \quad \bfT_1 (x) =  x.
\end{equation}
Then we have
\begin{equation}\label{eq:Tn(x)}
\bfT_n(x) = \left(\frac{x+\sqrt{x^2+4}}{2}\right)^n+\left(\frac{x-\sqrt{x^2+4}}{2}\right)^n.
\end{equation}
Let $q$,  $q'$ be positive integers with $q \ge q'$. Then \eqref{eq1} also implies that 
for $x, y >1$  
\begin{equation}\label{gap_ineq}
\begin{split}
\frac{\bfT_{q+q'}(x) - \bfT_{q+q'}(y)}{\bfT_{q}(x) - \bfT_{q}(y)} &=  
\bfT_{q}(x) \frac{\bfT_{q'}(x) - \bfT_{q'}(y) }{{\bfT_{q}(x) - \bfT_{q}(y)}} 
+ \bfT_{q'}(y) \frac{\bfT_{q}(x) - \bfT_{q}(y) }{{\bfT_{q}(x) - \bfT_{q}(y)}}  \\
&\qquad   - (-1)^{q'} \frac{\bfT_{q-q'}(x) - \bfT_{q-q'}(y)}{{\bfT_{q}(x) - \bfT_{q}(y)}}  \\
&< \bfT_{q}(x) + \bfT_{q'}(y)  +1. 
\end{split}
\end{equation}

Let $a_1, \ldots , a_n$ be positive integers. 
Since all the coefficients of $\bfT_n(x)$ are nonnegative
integers, we have $\bfT_n'(x)>0$ for all positive real numbers $x$.
If $n$ is even, then $T (a_1,\ldots,a_n)>2$ and $\bfT_n(0)=2$.
If $n$ is odd, then $T(a_1,\ldots,a_n)>0$ and $\bfT_n(0)=0$.
Thus, in any case, there is a unique positive $\mu$ such that $\bfT_n(\mu)=T(a_1,\ldots,a_n)$. This real number $\mu$ can be seen as being a mean of $a_1, \ldots , a_n$.

\begin{prop}\label{prop:phi}
If $\alpha$ is the quadratic irrational whose continued fraction expansion is 
given by $[a_0;a_1,\ldots,a_r,\overline{a_{r+1},\ldots,a_{r+s}}]$, then
$$
\cL(\alpha)  = \log \frac{ \mu + \sqrt{\mu^2+4}}{2},
$$
where $\mu$ is the positive real number such that 
$$  \bfT_s (\mu) = T(a_{r+1}, \dots , a_{r+s}).$$
\end{prop}

\begin{proof}

It follows from Lemma~\ref{lem:JaLi} and the definition of $\mu$ that 
\begin{equation*}\begin{split}
\cL(\alpha) & =  \frac{1}{s} \log\frac{T(a_{r+1}, \dots , a_{r+s}) +\sqrt{T(a_{r+1}, \dots , a_{r+s})^2-(-1)^s4}}{2} \\ 
& =  \frac{1}{s} \log\frac{\bfT_s( \mu)+\sqrt{\bfT_s (\mu)^2-(-1)^s4}}{2}.
\end{split}\end{equation*}
Setting
$$
r_1 = \frac{\mu+\sqrt{\mu^2+4}}{2}, \quad r_2 = \frac{\mu-\sqrt{\mu^2+4}}{2},  
$$
it follows from \eqref{eq:Tn(x)} that
$$
\bfT_s (\mu)^2 - (-1)^s 4  =  \left( r_1^s + r_2^s \right)^2 - 4 r_1^s r_2^s =  \left( r_1^s - r_2^s \right)^2
$$
and
$$
\frac{\bfT_s( \mu)+\sqrt{\bfT_s (\mu)^2-(-1)^s4}}{2} = r_1^s. 
$$
This establishes the proposition. 
\end{proof}

\section{Proof of Theorem \ref{thm2}}\label{sec:sec4}

In the sequel, the words are written over the alphabet $\{a, b\}$, where $a, b$ are integers with $1 \le a < b$. 
We recall the mechanical word $\bs_{p/q,\rho}$, defined in \eqref{eq:mechanical word}  is 
purely periodic with period $q$.   

\begin{defn}\label{def:Chr}
Let $p/q$ be a rational number in $[0, 1]$.
The \emph{lower Christoffel word} of slope $p/q$ is the prefix of $\bs_{p/q,0}$ 
 of length $q$. 
We denote it by $w_{p/q}$. 
\end{defn}

The following words are examples of lower Christoffel words: 
\begin{equation}\label{eq:ex Chr}
w_{0/1} = a, \ w_{1/1} = b, \ w_{1/2} = ab, \ w_{1/3} = aab, \ w_{2/5} = aabab, \cdots .
\end{equation}
%
We refer the reader to \cite{BLRS09} and \cite{Ai15} for additional results on Christoffel words.

For shorten the notation, for a finite word $v = v_1 \ldots v_n$  
over the positive integers, we write 
$$
[0 ; \overline{v}] = [0 ; \overline{v_1, \ldots , v_n}] \quad
\text{and} \quad
T(v) = T(v_1, \ldots , v_n). 
$$

We recall that we will show that
$$\{\cL([0;\bs]): \bs \text{ is a mechanical word over }\{a,b\} \: \} = \left[\cL([0;\overline{a}]),\cL([0;\overline{b}])\right].$$
To show it, 
we define the function $f$ on $[0, 1]$ by setting
\begin{equation}\label{eq:f}
f(\theta) =\cL (\alpha_{\theta}), \quad \hbox{for $\theta$ in $[0, 1]$,}
\end{equation}
where $\alpha_{\theta}$ is defined by
$$
\alpha_\theta  = [0; \bs_{\theta,0}].
$$
Note that if $\theta = p/q$, then $\bs_{\theta,0} = \overline{w_{p/q}}$.

Since Sturmian continued fractions and quadratic irrationals have L\'{e}vy constants, 
the function $f$ is well-defined on $[0,1]$.
Note that $f(0) = \cL ([0;  \overline{a}])$ and  $f(1) = \cL ([0;  \overline{b}])$,

We first give the structure of the proof of Theorem~\ref{thm2},  
before stating the propositions and lemmas we will use in the proof.

\begin{proof}[Proof of Theorem~\ref{thm2}]
Our aim is to prove that $f$ is increasing and continuous. 
Then we deduce $f([0,1]) = [\cL ([0;  \overline{a}]), \cL ([0;  \overline{b}])]$.

In Proposition~\ref{prop:monotone}, we will show that the function $f$ is monotone increasing on the rationals in $[0,1]$.
After then, we will see that $f$ is monotone increasing on the whole 
interval $[0,1]$ by using Lemma~\ref{lem:cont}. 

Furthermore, for an irrational number $\alpha$ in $[0,1]$, we have 
$$f(\alpha)  = \lim_{k\to\infty}f(p_{2k}/q_{2k}) = \lim_{k\to\infty}f(p_{2k+1}/q_{2k+1}),$$
 where the sequences of rational numbers 
 $(p_{2k}/q_{2k})_{k \ge 1}$ and $( p_{2k+1}/q_{2k+1} )_{k \ge 1}$ 
 are monotone increasing and monotone decreasing to $\alpha$, respectively. Thus $f$ has no jump discontinuity at $\alpha$.
Finally, in Lemma~\ref{lem:converge}, we will check that $f$ has no jump discontinuities at rational points. 
Therefore, $f$ is continuous on $[0,1]$. 
\end{proof}


We introduce further notation.
For any rational number $p/q$ in $[0, 1]$, 
we denote by $x_{p/q}$
the positive real solution of
\begin{equation}\label{eq:polyT}
\bfT_q( x_{p/q} ) =  T ( w_{p/q}).
\end{equation}
It has been shown just above Proposition \ref{prop:phi} that $x_{p/q}$ is well-defined.
Setting 
$$X_{p/q} = \begin{pmatrix} x_{p/q} & 1 \\ 1 & 0 \end{pmatrix},$$
it follows from \eqref{eq:polyT} that we have 
$$
\Tr(X_{p/q}^q) = T (w_{p/q}).
$$

\begin{exam}\label{ex:prop:phi}
\begin{enumerate}
\item For $n=1$, we have that 
$$
\mathcal L( [0; \overline{a}] ) =  \log \frac{a + \sqrt{a^2 + 4}}{2}
$$
and that $x_{0/1} =a$.    

\item For $n=2$, using Lemma~\ref{lem:JaLi}, we have
\begin{align*}
\mathcal L( [0; \overline{a,b}] )  
= \frac{1}{2} \log \frac{ab +2 + \sqrt{ (ab+2)^2 - 4}}{2}= \frac{1}{2} \log \frac{ab +2 + \sqrt{ ab (ab+4) }}{2}.
\end{align*}
Thus, we have
\begin{align*}
& \log \frac{\sqrt{ab} + \sqrt{ab+4}}{2}
= \frac 12 \log \left( \frac{\sqrt{ab} + \sqrt{ab+4}}{2}\right)^2 \\
& = \frac 12 \log \frac{ 2ab + 4 + 2 \sqrt{ab(ab+4)}}{4}  = \mathcal L( [0; \overline{a,b}] ),
\end{align*}
we check that 
$x_{1/2} = \sqrt{ab}$.   
\end{enumerate}
\end{exam}

In \eqref{eq:ex Chr}, we can observe  the following factorization property:  
$w_{1/2}=w_{0/1}w_{1/1}$, $w_{1/3}=w_{0/1}w_{1/2}$ and $w_{2/5}=w_{1/3}w_{1/2}$. This property is general and we use it to get the following lemma.

\begin{lem}\label{lem:second}
Let $p/q$ and $p'/q'$ be rational numbers in $[0, 1]$ 
with $\det \begin{pmatrix} p & p' \\ q & q' \end{pmatrix} = \pm 1$ and $q \ge q'$. 
Then we have
\begin{equation}\label{eq2}
T( w_{(p+p')/(q+q')}) = T( w_{p/q} ) T( w_{p'/q'}) - (-1)^{q'} T(w_{(p-p')/(q-q')})
\end{equation}
and
\begin{equation}\label{ineq23}
T( w_{(p+p')/(q+q')}) \ge T( w_{p/q} ) +1.
\end{equation}
Here $w_{\pm 1/0}$ denotes the empty word and $T (w_{\pm 1/0}) = \Tr(X_{\pm 1/0}^0) = \Tr(I) = 2.$
\end{lem}

\begin{proof}
It is known that every Christoffel word factors in two Christoffel words 
as follows (this is called the standard factorization): 
\begin{equation}\label{eq:standard factorization}
w_{(p+p')/(q+q')} = w_{p/q} w_{p'/q'}, \qquad w_{p/q} = w_{p'/q'} w_{(p-p')/(q-q')},
\end{equation}
see \cite[Theorem 2.4.1]{Re18} for the proof of the factorization; see also  
\cite[Chapter 3]{BLRS09} and \cite[Proof of Theorem 7.6]{Ai15}.
Therefore, if $q > q'$, then using Lemma~\ref{lem:first}, we obtain \eqref{eq2}.
Using the fact that for $a_1, \dots, a_k \in \{ a,b\}$ the matrix
$$\begin{pmatrix} a_1 & 1 \\ 1 & 0 \end{pmatrix} \begin{pmatrix} a_2 & 1 \\ 1 & 0 \end{pmatrix} \cdots \begin{pmatrix} a_k & 1 \\ 1 & 0 \end{pmatrix}
= \begin{pmatrix} a & b \\ c & d \end{pmatrix}$$
satisfies $a \ge b \ge d$ and $a \ge c \ge d$,  the inequality \eqref{ineq23} follows from \eqref{eq:standard factorization}.
%
%
If $q = q'$, then $p/q$ and $p'/q'$ are $0/1$ and $1/1$. It follows that
$$
T( w_{1/2} ) = T(ab) = ab+2 = T(a)T(b) + 2 = T( w_{1/1} ) T( w_{0/1}) + T(w_{1/0}).
$$
\end{proof}

\begin{exam}
The lower Christoffel words of slope $0/1$, $1/4$, $1/3$, $2/7$ are
$$w_{0/1} = a, \quad  w_{1/4} = aaab,  \quad w_{1/3} = aab, \quad w_{2/7} = aaabaab,$$
respectively.
Their corresponding traces are 
$$
T ( w_{0/1} )  = a, \quad T ( w_{1/4} )  = a^3 b + 2a^2 + 2ab + 2, \quad T ( w_{1/3} )  = a^2 b + 2a + b,
$$
$$
T ( w_{2/7} )  = a^5 b^2 + 4a^4 b + 3a^3b^2 + 4a^3 + 8 a^2 b + 2ab^2 + 5a + 2b.
$$
We check that we have 
$$
T ( w_{2/7} )  = T ( w_{1/4} ) T ( w_{1/3} ) + T ( w_{0/1} ),
$$
as given by the lemma. 
\end{exam}

\begin{lem}\label{lem:third}
Let $p/q$ and $p'/q'$ be rational numbers in $[0, 1]$ 
with $\det \begin{pmatrix} p & p' \\ q & q' \end{pmatrix} = \pm 1$ and $q > q'$.
Then we have the following four relations:
\begin{equation}\label{eqe}
\begin{split}
&\frac{\bfT_{q+q'}(x_{p/q})-\bfT_{q+q'}(x_{(p+p')/(q+q')})}{\bfT_{q'}(x_{p/q})-\bfT_{q'}(x_{p'/q'})}  \\
&\qquad \qquad \qquad = \bfT_{q}(x_{p/q}) + (-1)^{q'} \frac{\bfT_{q-q'}(x_{(p-p')/(q-q')})-\bfT_{q-q'}(x_{p/q})}{\bfT_{q'}(x_{p/q})-\bfT_{q'}(x_{p'/q'})},
\end{split}
\end{equation}
\begin{equation}\label{eqo}
\begin{split}
&\frac{\bfT_{q+q'}(x_{p'/q'})-\bfT_{q+q'}(x_{(p+p')/(q+q')})}{\bfT_{q}(x_{p'/q'})-\bfT_{q}(x_{p/q})} \\
& \qquad \qquad \qquad = \bfT_{q'}(x_{p'/q'}) - (-1)^{q'} \frac{\bfT_{q-q'}(x_{p'/q'})-\bfT_{q-q'}(x_{(p-p')/(q-q')})}{\bfT_{q}(x_{p'/q'})-\bfT_{q}(x_{p/q})}, 
\end{split}
\end{equation}
\begin{equation}\label{eqt}
\begin{split}
&\frac{\bfT_{q'}(x_{p'/q'}) - \bfT_{q'}(x_{(p+p')/(q+q')})}{\bfT_{q}(x_{(p+p')/(q+q')}) - \bfT_{q}(x_{p/q})} \\
& = \frac{\bfT_{q'}(x_{(p+p')/(q+q')})}{\bfT_{q}(x_{p/q})} - (-1)^{q'} \frac{\bfT_{q-q'}(x_{(p+p')/(q+q')})-\bfT_{q-q'}(x_{(p-p')/(q-q')})}{(\bfT_{q}(x_{(p+p')/(q+q')})-\bfT_{q}(x_{p/q})) \bfT_{q}(x_{p/q})}, 
\end{split}
\end{equation}
\begin{equation}\label{equ}
\begin{split}
&\frac{\bfT_{q}(x_{(p+p')/(q+q')})-\bfT_{q}(x_{p/q})}{\bfT_{q'}(x_{p'/q'})-\bfT_{q'}(x_{(p+p')/(q+q')})} \\
&= \frac{\bfT_{q}(x_{(p+p')/(q+q')})}{\bfT_{q'}(x_{p'/q'})} + (-1)^{q'} \frac{\bfT_{q-q'}(x_{(p+p')/(q+q')})-\bfT_{q-q'}(x_{(p-p')/(q-q')})}{(\bfT_{q'}(x_{p'/q'})-\bfT_{q'}(x_{(p+p')/(q+q')})) \bfT_{q'}(x_{p'/q'})}. 
\end{split}
\end{equation}
\end{lem}

\begin{proof}
By \eqref{eq2}, we have
\begin{equation}\label{eq3}
\bfT_{q+q'}(x_{(p+p')/(q+q')}) = \bfT_{q}(x_{p/q}) \bfT_{q'}(x_{p'/q'}) - (-1)^{q'} \bfT_{q-q'}(x_{(p-p')/(q-q')}).
\end{equation}
By applying \eqref{eq1} with $X = X_{p/q}$, $X= X_{p'/q'}$, and $X = X_{(p+p')/(q+q')}$, we get 
\begin{align*}
\bfT_{q+q'}(x_{p/q}) &= \bfT_{q}(x_{p/q}) \bfT_{q'}(x_{p/q}) - (-1)^{q'} \bfT_{q-q'}(x_{p/q}),\\
\bfT_{q+q'}(x_{p'/q'}) &= \bfT_{q}(x_{p'/q'}) \bfT_{q'}(x_{p'/q'}) - (-1)^{q'} \bfT_{q-q'}(x_{p'/q'}), \\
\bfT_{q+q'}(x_{(p+p')/(q+q')}) &= \bfT_{q}(x_{(p+p')/(q+q')}) \bfT_{q'}(x_{(p+p')/(q+q')}) - (-1)^{q'} \bfT_{q-q'}(x_{(p+p')/(q+q')}). 
\end{align*}
By combining these three equalities with \eqref{eq3}, we derive \eqref{eqe}, \eqref{eqo} and 
\begin{multline*}
\bfT_{q}(x_{(p+p')/(q+q')})  \bfT_{q'}(x_{(p+p')/(q+q')})
-\bfT_{q}(x_{p/q}) \bfT_{q'}(x_{p'/q'}) \\
=  (-1)^{q'} \left( \bfT_{q-q'}(x_{(p+p')/(q+q')})-\bfT_{q-q'}(x_{(p-p')/(q-q')}) \right),
\end{multline*}
from which \eqref{eqt} and \eqref{equ} follow.
\end{proof}

\begin{prop}\label{prop:monotone}
Let $p/q$ and $p'/q'$ be rational numbers in $[0, 1]$ 
with $\det \begin{pmatrix} p & p' \\ q & q' \end{pmatrix} = \pm 1$ and $q > q'$.
Then we have 
$$
x_{p/q} < x_{(p+p')/(q+q')} < x_{p'/q'} \ \text{ or } \ 
x_{p'/q'} < x_{(p+p')/(q+q')} < x_{p/q}. 
$$
\end{prop}

\begin{proof}
It is easy to check that, for every $n \ge 1$ the polynomial function $ \bfT_n ( x)$ is increasing.
Therefore, for every rational numbers $r/s,~ u/v$ in $[0, 1]$ and every $n \ge 1$, we have 
$$
\bfT_n ( x_{r/s} ) < \bfT_n ( x_{u/v} ) 
 \ \text{ if and only if } \ x_{r/s} < x_{u/v}.$$ 
We argue by induction. 
Under the assumption that $x_{p/q}$ lies between $x_{p'/q'}$ and $x_{(p-p')/(q-q')}$, we will show that $x_{(p+p')/(q+q')}$ lies between $x_{p/q}$ and $x_{p'/q'}$.
It will be achieved by showing that
\begin{equation}\label{direction1}
\frac{\bfT_{q+q'}( x_{p/q}) -\bfT_{q+q'}( x_{(p+p')/(q+q')}) }{ \bfT_{q'}(x_{p/q}) -\bfT_{q'}( x_{{p'}/{q'}}) } >0
\end{equation}
and
\begin{equation}\label{direction2}
\frac{\bfT_{q+q'}(x_{p'/q'})-\bfT_{q+q'}(x_{(p+p')/(q+q')})}{\bfT_{q}(x_{p'/q'})-\bfT_{q}(x_{p/q})} >0.
\end{equation}
We first show  \eqref{direction1} through (I) and (II-a) to (II-c).

(I) Suppose that $q'$ is even.
Since $x_{p/q}$ is between $x_{p'/q'}$ and $x_{(p-p')/(q-q')}$, we have
$$
\frac{\bfT_{q-q'}(x_{p/q})-\bfT_{q-q'}(x_{(p-p')/(q-q')})}{\bfT_{q'}(x_{p/q})-\bfT_{q'}(x_{p'/q'})} < 0.
$$
By \eqref{eqe},
\begin{equation}\label{ineq1}
\frac{\bfT_{q+q'}(x_{p/q})-\bfT_{q+q'}(x_{(p+p')/(q+q')})}{\bfT_{q'}(x_{p/q})-\bfT_{q'}(x_{p'/q'})} > \bfT_{q}(x_{p/q}) >0.
\end{equation}

Next, we assume that $q'$ is odd from (II-a) to (II-c).

(II-a) If  $q-q' = q'$, then $p/q = 1/2$ and $p'/q' = 0/1$ or $1/1$.
For $p'/q' = 0/1$, we have 
\begin{equation}\label{ineq01}
\begin{split}
&\frac{\bfT_{q-q'}(x_{p/q})-\bfT_{q-q'}(x_{(p-p')/(q-q')})}{\bfT_{q'}(x_{p'/q'})-\bfT_{q'}(x_{p/q})}  \\
& =\frac{\bfT_{1}(x_{1/2})-\bfT_{1}(x_{1/1})}{\bfT_{1}(x_{0/1})-\bfT_{1}(x_{1/2})} = \frac{\sqrt{ab}-b}{a-\sqrt{ab}}=\sqrt{\frac ba} < ab+2 = \bfT_2(x_{1/2}).
\end{split}
\end{equation}
For $p'/q' = 1/1$ we have 
\begin{equation}\label{ineq11}
\begin{split}
&\frac{\bfT_{q-q'}(x_{p/q})-\bfT_{q-q'}(x_{(p-p')/(q-q')})}{\bfT_{q'}(x_{p'/q'})-\bfT_{q'}(x_{p/q})}  \\
& =\frac{\bfT_{1}(x_{1/2})-\bfT_{1}(x_{0/1})}{\bfT_{1}(x_{1/1})-\bfT_{1}(x_{1/2})}= \frac{\sqrt{ab}-a}{b-\sqrt{ab}}=\sqrt{\frac ab} <  ab+2 =  \bfT_2(x_{1/2}).
\end{split}
\end{equation}

(II-b) If $q-q' > q'$,  then  $x_{(p-p')/(q-q')}$ lies between $x_{p'/q'}$ and $x_{(p-2p')/(q-2q')}$, thus,
$$
 \frac{ \bfT_{q-2q'}(x_{p/q}) - \bfT_{q-2q'}(x_{(p-2p')/(q-2q')}) }{\bfT_{q'}(x_{p'/q'})-\bfT_{q'}(x_{p/q})} >0.
$$
Thus by \eqref{equ},
\begin{equation}\label{ineq2iv}
\begin{split}
& \frac{\bfT_{q-q'}(x_{p/q})-\bfT_{q-q'}(x_{(p-p')/(q-q')})}{\bfT_{q'}(x_{p'/q'})-\bfT_{q'}(x_{p/q})} \\
& \qquad\qquad\qquad =  \frac{\bfT_{q-q'}(x_{p/q})}{\bfT_{q'}(x_{p'/q'})} - \frac{ \bfT_{q-2q'}(x_{p/q}) - \bfT_{q-2q'}(x_{(p-2p')/(q-2q')}) }{(\bfT_{q'}(x_{p'/q'})-\bfT_{q'}(x_{p/q})) \bfT_{q'}(x_{p'/q'})}\\
& \qquad\qquad\qquad < \frac{\bfT_{q-q'}(x_{p/q})}{\bfT_{q'}(x_{p'/q'})}< \bfT_{q-q'}(x_{p/q}) < \bfT_{q}(x_{p/q}).
 \end{split}
\end{equation}

(II-c) If  $q-q' < q'$, then $x_{p'/q'}$ lies between $x_{(p-p')/(q-q')}$ and $x_{(2p'-p)/(2q'-q)}$, thus,
\begin{equation}\label{caseiic-ineq}
\frac{\bfT_{2q'-q}(x_{(2p'-p)/(2q'-q)})-\bfT_{2q'-q}(x_{p'/q'})}{\bfT_{q-q'}(x_{p'/q'})-\bfT_{q-q'}(x_{(p-p')/(q-q')})} >0
\end{equation}
and
\begin{equation}\label{caseiic-ineq2}
\frac{\bfT_{2q'-q}(x_{p/q})-\bfT_{2q'-q}(x_{(2p'-p)/(2q'-q)})}{\bfT_{q'}(x_{p/q})-\bfT_{q'}(x_{p'/q'}) } >0 .
\end{equation}
For odd $q$,  we have,  by \eqref{eqt} and \eqref{caseiic-ineq2}, 
\begin{equation}\label{ineq2ii}
\begin{split}
& \frac{\bfT_{q-q'}(x_{(p-p')/(q-q')})-\bfT_{q-q'}(x_{p/q})}{\bfT_{q'}(x_{p/q})-\bfT_{q'}(x_{p'/q'})} \\
&\qquad\qquad\qquad  = \frac{\bfT_{q-q'}(x_{p/q})}{\bfT_{q'}(x_{p'/q'})} - \frac{\bfT_{2q'-q}(x_{p/q})-\bfT_{2q'-q}(x_{(2p'-p)/(2q'-q)})}{(\bfT_{q'}(x_{p/q})-\bfT_{q'}(x_{p'/q'})) \bfT_{q'}(x_{p'/q'})} \\
 &\qquad\qquad\qquad  < \frac{\bfT_{q-q'}(x_{p/q})}{\bfT_{q'}(x_{p'/q'})} < \bfT_{q-q'}(x_{p/q})  < \bfT_{q}(x_{p/q}).
 \end{split}
\end{equation}

Now we assume that $q$ is even.  
If $q-q' < 2q' - q$,  then $x_{(2p'-p)/(2q'-q)}$  
lies between $x_{(p-p')/(q-q')}$ and $x_{(3p'-2p)/(3q'-2q)}$.
By \eqref{eqt}
\begin{equation*}
\begin{split}
&\frac{\bfT_{q-q'}(x_{(p-p')/(q-q')}) - \bfT_{q-q'}(x_{p'/q'})}{\bfT_{2q'-q}(x_{p'/q'}) - \bfT_{2q'-q}(x_{(2p'-p)/(2q'-q)})} -\frac{\bfT_{q-q'}(x_{p'/q'})}{\bfT_{2q'-q}(x_{(2p'-p)/(2q'-q)})}  \\
& = \frac{\bfT_{3q'-2q}(x_{p'/q'})-\bfT_{3q'-2q}(x_{(3p'-2p)/(3q'-2q)})}{(\bfT_{2q'-q}(x_{p'/q'})-\bfT_{2q'-q}(x_{(2p'-p)/(2q'-q)})) \bfT_{2q'-q}(x_{(2p'-p)/(2q'-q)})} >0.
\end{split}
\end{equation*}
If $q-q' > 2q' - q$, then $x_{(p-p')/(q-q')}$ 
lies between $x_{(2p'-p)/(2q'-q)}$ and $x_{(2p-3p')/(2q-3q')}$.
By \eqref{equ}
\begin{equation*}
\begin{split}
&\frac{\bfT_{q-q'}(x_{p'/q'})-\bfT_{q-q'}(x_{(p-p')/(q-q')})}{\bfT_{2q'-q}(x_{(2p'-p)/(2q'-q)})-\bfT_{2q'-q}(x_{p'/q'})} - \frac{\bfT_{q-q'}(x_{p'/q'})}{\bfT_{2q'-q}(x_{(2p'-p)/(2q'-q)})}\\
&= \frac{\bfT_{2q-3q'}(x_{p'/q'})-\bfT_{2q-3q'}(x_{(2p-3p')/(2q-3q')})}{(\bfT_{2q'-q}(x_{(2p'-p)/(2q'-q)})-\bfT_{2q'-q}(x_{p'/q'})) \bfT_{2q'-q}(x_{(2p'-p)/(2q'-q)})} >0.
\end{split}
\end{equation*}
Since $q-q'$ is odd and $2q'-q$ is even, we have $q-q' \ne 2q'-q$. Therefore, for the both cases of $q-q'<2q'-q$ and $q-q'>2q'-q$, we conclude that 
\begin{equation}\label{ineq3}
\frac{\bfT_{q-q'}(x_{(p-p')/(q-q')}) - \bfT_{q-q'}(x_{p'/q'})}{\bfT_{2q'-q}(x_{p'/q'}) - \bfT_{2q'-q}(x_{(2p'-p)/(2q'-q)})} > \frac{\bfT_{q-q'}(x_{p'/q'})}{\bfT_{2q'-q}(x_{(2p'-p)/(2q'-q)})}.
\end{equation}

By \eqref{eqe}, we have
\begin{multline*}
\frac{\bfT_{q}(x_{p'/q'})-\bfT_{q}(x_{p/q})}{\bfT_{q-q'}(x_{p'/q'})-\bfT_{q-q'}(x_{(p-p')/(q-q')})}  \\
= \bfT_{q'}(x_{p'/q'}) - \frac{\bfT_{2q'-q}(x_{(2p'-p)/(2q'-q)})-\bfT_{2q'-q}(x_{p'/q'})}{\bfT_{q-q'}(x_{p'/q'})-\bfT_{q-q'}(x_{(p-p')/(q-q')})},
\end{multline*}
which is equivalent to
\begin{multline}\label{eqep}
 \bfT_{q-q'}(x_{(p-p')/(q-q')}) - \bfT_{q-q'}(x_{p'/q'}) \\
= \frac{\bfT_{q}(x_{p/q}) - \bfT_{q}(x_{p'/q'}) + \bfT_{2q'-q}(x_{p'/q'}) - \bfT_{2q'-q}(x_{(2p'-p)/(2q'-q)})}{\bfT_{q'}(x_{p'/q'})}.
\end{multline}
By plugging \eqref{eqep} into \eqref{ineq3}, we deduce that
\begin{equation}\label{ineq4}
\begin{split}
\frac{\bfT_{q}(x_{p'/q'})-\bfT_{q}(x_{p/q})}{\bfT_{2q'-q}(x_{(2p'-p)/(2q'-q)})-\bfT_{2q'-q}(x_{p'/q'})} 
&> \frac{ \bfT_{q'}(x_{p'/q'}) \bfT_{q-q'}(x_{p'/q'})}{\bfT_{2q'-q}(x_{(2p'-p)/(2q'-q)})}  -1\\
&\ge \frac{ \bfT_{q'}(x_{p'/q'}) \bfT_{q-q'}(x_{p'/q'})}{\bfT_{q'}(x_{p'/q'}) -1 }  -1 \\
&= \frac{ \bfT_{q'}(x_{p'/q'})(\bfT_{q-q'}(x_{p'/q'})-1)+1}{\bfT_{q'}(x_{p'/q'}) -1 } \\
&\ge \frac{1}{\bfT_{q'}(x_{p'/q'}) -1 }, 
\end{split}
\end{equation}
where the second inequality follows from \eqref{ineq23}.    
Therefore, by \eqref{gap_ineq} and \eqref{ineq4}
\begin{equation}\label{ineq5}
\begin{split}
&\frac{\bfT_{2q'-q}(x_{p/q})-\bfT_{2q'-q}(x_{(2p'-p)/(2q'-q)})}{\bfT_{q'}(x_{p/q})-\bfT_{q'}(x_{p'/q'}) } \\
&= \frac{\bfT_{2q'-q}(x_{p/q})-\bfT_{2q'-q}(x_{p'/q'})}{\bfT_{q'}(x_{p/q})-\bfT_{q'}(x_{p'/q'}) } \\
&\qquad + \frac{\bfT_{q}(x_{p/q})-\bfT_{q}(x_{p'/q'}) }{\bfT_{q'}(x_{p/q})-\bfT_{q'}(x_{p'/q'}) }  \frac{\bfT_{2q'-q}(x_{p'/q'})-\bfT_{2q'-q}(x_{(2p'-p)/(2q'-q)})}{\bfT_{q}(x_{p/q})-\bfT_{q}(x_{p'/q'}) } \\
&< 1 +\left( \bfT_{q-q'}(x_{p/q}) + \bfT_{q'}(x_{p'/q'}) +1 \right)  (\bfT_{q'}(x_{p'/q'}) -1) \\
&= \bfT_{q-q'}(x_{p/q}) \left( \bfT_{q'}(x_{p'/q'}) -1 \right) + \bfT_{q'}(x_{p'/q'}) ^2.
\end{split}
\end{equation}
Using \eqref{eqt} and \eqref{ineq5}, we have
\begin{equation}\label{ineq2iii}
\begin{split}
& \frac{\bfT_{q-q'}(x_{p/q})-\bfT_{q-q'}(x_{(p-p')/(q-q')})}{\bfT_{q'}(x_{p'/q'})-\bfT_{q'}(x_{p/q})} \\
& = \frac{\bfT_{q-q'}(x_{p/q})}{\bfT_{q'}(x_{p'/q'})} + \frac{\bfT_{2q'-q}(x_{p/q})-\bfT_{2q'-q}(x_{(2p'-p)/(2q'-q)})}{(\bfT_{q'}(x_{p/q})-\bfT_{q'}(x_{p'/q'})) \bfT_{q'}(x_{p'/q'})} \\ 
&< \frac{\bfT_{q-q'}(x_{p/q})}{\bfT_{q'}(x_{p'/q'})}  
+ \frac{\bfT_{q-q'}(x_{p/q}) \left( \bfT_{q'}(x_{p'/q'}) -1 \right) + \bfT_{q'}(x_{p'/q'}) ^2 } { \bfT_{q'}(x_{p'/q'}) } \\
&= \bfT_{q-q'}(x_{p/q})  + \bfT_{q'}(x_{p'/q'}) .
\end{split}
\end{equation}
To show $\bfT_{q-q'}(x_{p/q})+\bfT_{q'}(x_{p'/q'})<\bfT_q(x_{p/q})$, we will consider the following two cases.
If $q-q' \ge 3$, then  $\bfT_{q-q'}(x) \ge 4$ for all $x \ge 1$.
Thus, from \eqref{eq1} and \eqref{eq2}, we have
\begin{equation}
\begin{split}
\bfT_{q}(x_{p/q}) &= \frac 12 \left(  \bfT_{q'}(x_{p/q}) \bfT_{q-q'}(x_{p/q}) + \bfT_{2q'-q}(x_{p/q}) \right) \\
&\qquad + \frac 12 \left( \bfT_{q'}(x_{p'/q'}) \bfT_{q-q'}(x_{(p-p')/(q-q')}) + \bfT_{2q'-q}(x_{(2p'-p)/(2q'-q)}) \right) \\
&>  \bfT_{q'}(x_{p/q}) + \bfT_{q'}(x_{p'/q'}) > \bfT_{q-q'}(x_{p/q}) + \bfT_{q'}(x_{p'/q'}).
\end{split}
\end{equation}
If $q-q' =1$, then since $2q'-q \ge 2$, we have 
$$
\bfT_{2q'-q}(x_{(2p'-p)/(2q'-q)}) \ge b = \bfT_1(x_{1/1}).
$$ 
Thus, from \eqref{eq2}, we have
\begin{equation}
\begin{split}
\bfT_{q}(x_{p/q}) &= \bfT_{q'}(x_{p'/q'}) \bfT_{q-q'}(x_{(p-p')/(q-q')}) + \bfT_{2q'-q}(x_{(2p'-p)/(2q'-q)}) \\
&\ge  \bfT_{q'}(x_{p'/q'}) +\bfT_{1}(x_{1/1}) \ge  \bfT_{q'}(x_{p'/q'}) + \bfT_{q-q'}(x_{p/q}).
\end{split}
\end{equation}
Therefore, using \eqref{ineq01}, \eqref{ineq11}, \eqref{ineq2iv}, 
\eqref{ineq2ii}, \eqref{ineq2iii}, 
we deduce from \eqref{eqe} that 
\begin{equation*}\label{ineq2}
\frac{\bfT_{q+q'}(x_{p/q})-\bfT_{q+q'}(x_{(p+p')/(q+q')})}{\bfT_{q'}(x_{p/q})-\bfT_{q'}(x_{p'/q'})} >0.
\end{equation*}
Thus, we have established \eqref{direction1}
regardless of the parity of $q'$.

We next show \eqref{direction2} through (III-a) to (III-b).

(III-a) Suppose that $q'$ is odd.
Since $x_{p/q}$ is between $x_{p'/q'}$ and $x_{(p-p')/(q-q')}$, we have
$$
\frac{\bfT_{q-q'}(x_{p'/q'})-\bfT_{q-q'}(x_{(p-p')/(q-q')})}{\bfT_{q}(x_{p'/q'})-\bfT_{q}(x_{p/q})} > 0.
$$
By \eqref{eqo}, 
\begin{equation*}
\frac{\bfT_{q+q'}(x_{p'/q'})-\bfT_{q+q'}(x_{(p+p')/(q+q')})}{\bfT_{q}(x_{p'/q'})-\bfT_{q}(x_{p/q})} > \bfT_{q'}(x_{p'/q'})>0.
\end{equation*}

(III-b) Suppose that $q'$ is even.
Let $m \ge 1$ be the integer satisfying that $m q' < q < (m+1) q'$.
Then $x_{p'/q'}$ is between $x_{((m+1)p'-p)/((m+1)q'-q)}$ and $x_{(p-mp')/(q-mq')}$.
For each $n = 0,1,\dots m$, set
$$A_n =  \frac{\bfT_{q-(n-1)q'}(x_{p'/q'})-\bfT_{q-(n-1)q'}(x_{(p-(n-1)p')/(q-(n-1)q')})}{\bfT_{q-nq'}(x_{p'/q'})-\bfT_{q-nq'}(x_{(p-n p')/(q-nq')})}$$
and we have $A_n>0$ for $1\le n\le m$.
By \eqref{eqe} we have (note that $q$ is odd)
\begin{align*}
A_m &= \bfT_{q'}(x_{p'/q'}) + \frac{ \bfT_{(m+1)q'-q}(x_{p'/q'}) -\bfT_{(m+1)q'-q}(x_{((m+1)p'-p)/((m+1)q'-q)})}{\bfT_{q-mq'}(x_{p'/q'})-\bfT_{q-mq'}(x_{(p-mp')/(q-mq')})}  \\
 &< \bfT_{q'}(x_{p'/q'}).
\end{align*}
Here, we have used the fact that $p'/q'$ is between $\frac{(m+1)p'-p}{(m+1)q'-q}$ and $\frac{p-mp'}{q-mq'}$.
By \eqref{eqo}, for $n=0, \ldots , m-1$, we have 
$$ A_n = \bfT_{q'}(x_{p'/q'}) - \frac{1}{A_{n+1}}$$ 
Therefore, inductively, we have 
\begin{equation*}\label{eqinq6}
\begin{split}
A_0 
> \bfT_{q'}(x_{p'/q'}) - \cfrac{1}{\bfT_{q'}(x_{p'/q'}) - \cfrac{1}{\ddots - \cfrac{1}{ \bfT_{q'}(x_{p'/q'}) } } } > 0.
\end{split}\end{equation*}
Since 
$$A_0 = \frac{\bfT_{q+q'}(x_{p'/q'})-\bfT_{q+q'}(x_{(p+p')/(q+q')})}{\bfT_{q}(x_{p'/q'})-\bfT_{q}(x_{p/q})}, $$
we have established \eqref{direction2}
regardless of the parity of $q'$.  We conclude that
$$
x_{p'/q'} < x_{(p+p')/(q+q')} < x_{p/q}   \ \text { or } \ x_{p/q} < x_{(p+p')/(q+q')} < x_{p'/q'}  
$$   
always holds.  
\end{proof}

The Farey sequence $\{ \mathcal F_N \}_{N \in \mathbb N}$
is the sequence of the set of rational numbers in $[0,1]$ whose denominator is less than or equal to $N$.
For instance, $\mathcal F_4$ is the set $\{ \frac 01, \frac 14, \frac 13,  \frac 12,  \frac 23, \frac 34, \frac 11\}$.
We say that two fractions $p/q$ and $p'/q'$ in $\mathcal F_N$ are neighbor if there is no element of $\mathcal F_N$ between $p/q$ and $p'/q'$.
It is well known that for any neighboring fractions $p/q$ and $p'/q'$ in $\mathcal F_N$ we have $| pq'-qp'| = 1$.

\begin{prop}\label{Prop:monotonicity}
Let $p/q$ and $p'/q'$ be rational numbers in $[0, 1]$ with $p/q < p'/q'$.
Then
$$
\mathcal L ( \alpha_{p/q} )  < \mathcal L ( \alpha_{p'/q'} ).
$$
\end{prop}

\begin{proof}
By considering the Farey sequence,
it is sufficient to show the conclusion for rational numbers $p/q$ and $p'/q'$ 
with $\det \begin{pmatrix} p & p' \\ q & q' \end{pmatrix} = - 1$.   
By Proposition~\ref{prop:phi},
$$
\mathcal L ( \alpha_{p/q} )  < \mathcal L ( \alpha_{(p+p')/(q+q')} ) < \mathcal L ( \alpha_{p'/q'} ),
$$
is equivalent to 
$$
x_{p/q}  < x_{(p+p')/(q+q')} < x_{p'/q'},
$$
which is established in Proposition~\ref{prop:monotone}.
For a rational number $p/q$ in $[0,1]$ with $q > 1$, we can write 
$$ \frac pq = \frac{p'+p''}{q'+q''},
$$
where $p'/q'$, $p''/q''$ are two distinct neighbors of $p/q$ in $\mathcal F_q$. 
Thus by an inductive argument starting with $x_{0/1} = a < x_{1/1} = b$, we complete the proof.
\end{proof}

The above proposition shows that $f$ is monotone increasing on the rationals.
From now on, we will discuss the continuity of $f$.

\begin{lem}\label{lem:cont}
For an irrational slope $\theta$, we have
$$\cL(\alpha_\theta) = \lim_{k\to \infty} \cL(\alpha_{p_{k}/q_{k}}),$$
where $p_{k}/q_{k}$ is the $k$-th principal convergent of $\theta$.   
\end{lem}


\begin{proof}    
By 
 \eqref{Levylimit}, we have 
\begin{equation*}
\cL(\alpha_\theta) = \lim_{n\to \infty}\frac{1}{n}\sum_{i=1}^n \log{([s_i;s_{i+1},\ldots])},
\end{equation*}
where $\bs_{\theta,0} = s_1 s_2 \ldots$.
It is known that $M_k = wab$ or $M_k=wba$ for some palindrome word $w$ and $w_{p_k/q_k} = awb$; see \cite[Theorem 8.40]{Ai15}.
Thus, for each $k \ge1$, we have that $w_{p_k/q_k} = s_1 \ldots s_{q_k -1} b$.    
The $q_{k}$-th letter $s_{q_k}$ of $\bs_{\theta,0}$ may not be $b$.
Thus, we have
\begin{align*}
& \cL(\alpha_{p_k/q_k}) 
 = \cL([0;\overline{s_1\ldots s_{q_k-1} b}]) \\
& = \frac{1}{q_k}\left[\log([\overline{b;s_1,s_2,\cdots,s_{q_k-1}}]) + \sum_{i = 1}^{q_k-1}\log{([\overline{s_i;s_{i+1},\ldots, s_{q_k-1}, b, s_{1}, \ldots, s_{i-1}}])}
\right].
\end{align*} 

For $r \ge 0$, let $\gamma = [c_{0}; c_{1}, c_{2}, \ldots]$, $\delta = [d_{0}; d_{1},d_{2},\ldots]$ and $R_{r}/S_{r} = [b_0;b_1,\cdots, b_{r}]$ be 
continued fractions with partial quotients in $\{a, b\}$. 
By using the fact that $S_{r}\ge 2^{\frac{r-1}{2}}$, we have
\begin{equation*}\begin{split}
& |[b_0;b_1,\ldots,b_{r},c_{0},c_{1},\ldots] - [b_0;b_1,\ldots,b_{r},d_{0},d_{1},\ldots] | \\
& = \left| \frac{R_{r}\gamma + R_{r-1}}{S_{r}\gamma + S_{r-1}} - \frac{R_{r}\delta + R_{r-1}}{S_{r}\delta + S_{r-1}} \right|  
=  \frac{|\gamma - \delta|}{(S_{r}\gamma + S_{r-1})(S_{r}\delta + S_{r-1})}   \\ 
& \le \frac{ [b ;\overline{a, b}] - [a ;\overline{b, a}] }{2^{r-1}}.  
\end{split}\end{equation*}

Since $|\log x - \log y | < |x-y|$ if $x,y>1$, we get
\begin{equation*}\begin{split}
&\left|\cL(\alpha_\theta)-\cL(\alpha_{p_k/q_k})\right| \\
&\le \frac{1}{q_k} \Bigg[ \Big| \log([\overline{s_{q_k};s_1,s_2,\cdots,s_{q_k-1}}]) - \log([\overline{b;s_1,s_2,\cdots,s_{q_k-1}}])\Big| \\
&\qquad  + \sum_{i=1}^{q_k-1} \Big| \log{([s_i;s_{i+1},\ldots])} - \log{([\overline{s_i;s_{i+1},\ldots, s_{q_k-1}, b, s_{1}, \ldots, s_{i-1}}])} \Big| \Bigg] \\
& \le \frac{[b;\overline{a, b}]-[a;\overline{b, a}]}{q_k} \left( 1 + \sum_{r=0}^{q_k-2}\frac{1}{2^{r-1}}  \right)
< \frac{ 5 ([b;\overline{a, b}]-[a;\overline{b, a}] ) }{q_k} 
\to 0 \ \as \ k\to \infty.
\end{split}\end{equation*}
We conclude that $\cL(\alpha_\theta) = \lim\limits_{k \to \infty}\cL(\alpha_{p_k/q_k})$.     
\end{proof}



In the following lemma, we show that $f$ has no jump discontinuity at  a rational number $p/q$ in $[0, 1]$.

\begin{lem}\label{lem:converge}
For a given rational number $p/q$ in $[0,1)$, there exist a decreasing sequence $(r_n)_{n \ge 1}$ and an increasing sequence $(r_n')_{n\ge1}$  of rational numbers which converge to $p/q$ such that $x_{r_n}$ and $x_{r_n'}$ converge to $x_{p/q}$.
\end{lem}

\begin{proof} 
Let $p/q$ be a rational number in  $[0,1]$. 
Choose a rational number $p'/q'$ such that $| p'q-pq' | = 1$.
For $n \ge 1$, setting $$r_n:=\frac{p'+np}{q'+nq},$$ 
we observe that $r_n$ is decreasing to $p/q$ or increasing to $p/q$ as $n$ tends to infinity,
according as $p'q-pq' = 1$ or $-1$.  
Note that $r_n\not=p/q$, for $n \ge 1$.

For brevity, we define a function $\varphi$ by setting
\begin{equation}\label{eq:phi}
\varphi (x) = \frac{x + \sqrt{x^2+4} }{2}, 
\qquad x \in \bR.
\end{equation}
By \eqref{eq:Tn(x)}, 
we have
\begin{equation}\label{eq:t_m1}
\bfT_m(x)= \varphi(x) ^m + \left(- \frac{1}{\varphi(x)} \right)^{m}, \qquad \text{for $m \ge 1$},
\end{equation}
and $\mathcal L ( \alpha_{p/q} ) = \log \varphi (x_{p/q} ).$
By Lemma~\ref{lem:second}, we get 
\begin{equation*}
T (w_{r_n})= T(w_{p/q}) T (w_{r_{n-1}})+(-1)^{q+1} T (w_{r_{n-2}}), \quad \hbox{for $n \ge 2$}.
\end{equation*}
Thus, the sequence $(T (w_{r_n}))_{n \ge 1}$ is a recurrence sequence and 
there exist constants $C_1,C_2$ such that
\begin{equation}\label{eq:t_m2}
T(w_{r_n})=C_1{u}^n+C_2{v}^n, \quad n \ge 1,
\end{equation}
where 
$$u=\frac{T(w_{p/q})+\sqrt{T(w_{p/q})^2+(-1)^{q+1}4}}{2}
~\andd~
v=\frac{T(w_{p/q})-\sqrt{T(w_{p/q})^2+(-1)^{q+1}4}}{2}.$$
Since the integer $T(w_{p/q})^2+(-1)^{q+1}4$ cannot be a positive perfect square, we 
deduce that $C_1$ and $C_2$ are nonzero. 
Since $$\frac 1q \log u = \cL(\alpha_{p/q}) =\log \vphi(x_{p/q}), $$
we have, for any $\veps >0$, 
$$
0 < \vphi(x_{p/q}-\veps)^q < u < \vphi(x_{p/q}+\veps)^q,
$$
if $q$ is large enough.
By \eqref{eq:t_m1}, \eqref{eq:t_m2} and \eqref{eq:polyT}, we get
\begin{equation*}\begin{split}
\cfrac{\bfT_{q'+nq}(x_{p/q} \pm \veps)}{\bfT_{q'+nq}(x_{r_n})} 
&= \frac{\vphi(x_{p/q} \pm \veps)^{q'+nq} + (-\vphi(x_{p/q} \pm \veps))^{-q'-nq}}{C_1u^n+C_2v^n} \\ 
&= \frac{\vphi(x_{p/q} \pm \veps)^{q'} \left(\frac{\vphi(x_{p/q} \pm \veps)^{q}}{u}\right)^n +  u^{-n} (-\vphi(x_{p/q} \pm \veps))^{-q'-nq} }{C_1 +C_2 \left(\frac{v}{u}\right)^n} .
\end{split}\end{equation*}
Since $\vphi (x) > 1$ for $x  > 0$ and $|v| < 1 < u$, 
we get, for small $\varepsilon >0$, that
\begin{equation*}\begin{split}
\cfrac{\bfP_{q'+nq}(x_{p/q}-\veps)}{\bfP_{q'+nq}(x_{r_n})} \to 0 ~\andd~
\cfrac{\bfP_{q'+nq}(x_{p/q}+\veps)}{\bfP_{q'+nq}(x_{r_n})} \to \infty,
\end{split}\end{equation*}
as $n\to \infty$.
There exists $N$ depending only on $\veps$ such that, for any $n>N$, 
$$\bfP_{q'+nq}(x_{p/q}-\veps) < \bfP_{q'+nq}(x_{r_n}) < \bfP_{q'+nq}(x_{p/q}+\veps).$$
Since $\bfP_{q'+nq}$ is monotone increasing, we obtain that 
$$x_{p/q}-\veps < x_{r_n} < x_{p/q}+\veps.$$
Since $\veps$ can be taken arbitrarily small, this shows that
$x_{r_n}$ tends to $x_{p/q}$ as $n$ tends to infinity. 
\end{proof}

\section*{Acknowledgement}
The authors wish to warmly thank the anonymous referee for many very helpful suggestions that 
help them to improve the readability of the paper. 

The research was supported by the French--Korean bilateral STAR project (NRF-2017K1A3A1A21013650), the NRF of Korea (NRF-2018R1A2B6001624), and the French--Korean LIA in Mathematics.
S. L. acknowledges the support of the Centro di Ricerca Matematica Ennio de Giorgi and of UniCredit Bank R\&D group for financial support through the Dynamics and Information Theory Institute at the Scuola Normale Superiore.

\end{document}